\documentclass[12pt]{article}
\usepackage[citecolor=blue]{hyperref}
\usepackage{amsmath,amsfonts,amssymb,amsthm}
\usepackage{mathrsfs}
\usepackage{verbatim}           
\usepackage{enumerate}
\usepackage[pdftex]{graphicx}
\usepackage[pdftex]{color}         
\usepackage{amscd}
\usepackage[footnotesize]{caption}

\usepackage{epstopdf}

\newtheorem{theorem}{Theorem}[section]
\newtheorem{lemma}[theorem]{Lemma}
\newtheorem{corollary}[theorem]{Corollary}
\newtheorem{proposition}[theorem]{Proposition}
\newtheorem{definition}[theorem]{Definition}

\theoremstyle{remark}
\newtheorem{remark}[theorem]{Remark}

\numberwithin{equation}{section}

\def\R{{\mathbb R}}

\def\Rplus{{\mathbb{R}_+}}

\def\|{{|\!|}}

\title{Transmission Eigenvalues and  Thermoacoustic Tomography}
\author{David Finch$^{\thanks{
         Department of Mathematics, Oregon State University, Corvallis
         OR, 97331, finch@math.oregonstate.edu, Supported in part by NSF Grant
         DMS-100914}}$ \\
    Kyle S. Hickmann$^{\thanks{Department of Mathematics, Tulane
        University, New Orleans, LA 70118,
         khickma@tulane.edu}}$ \\
}

\begin{document}
\maketitle

\begin{abstract}
  The spectrum of the interior transmission problem is related to the
  unique determination of the acoustic properties of a body in
  thermoacoustic imaging. Under a non-trapping hypothesis, we show
  that sparsity of the interior transmission spectrum implies a range
  separation condition for the thermoacoustic operator.  In odd
  dimensions greater than or equal to three, we prove that the
  interior transmission spectrum for a pair of radially symmetric
  non-trapping sound speeds is countable, and conclude that the ranges
  of the associated theromacoustic maps have only trivial
  intersection.
\end{abstract}


\medskip


\medskip


\section{Introduction}

The aims of this paper are to point out a connection between the
interior transmission eigenvalue spectrum and a type of uniqueness
question for sound speed in a  wave equation and to use this
connection to make some conclusions about the wave equation.  The wave
equation problem we consider arises in thermoacoustic tomography (TAT).
In the standard model of TAT, a pressure wave  is generated in a
body  $D \subset \mathbb{R}^d$ whose sound speed is a pertubation of a
constant background sound speed which we will take throughout to be
unity. We take $u(x,t)$ to be the solution of  the Cauchy problem
\begin{align}\label{forward_TAT_problem}
\partial^2_t u(x,t) - c^2(x) \Delta u(x,t) &= 0 \textrm{ on }
\mathbb{R}^d \times \Rplus \\ \nonumber u(x,0) = f(x), \, \partial_t
u(x,0) &= 0 \textrm{ for } x \in \mathbb{R}^d
\end{align} 
for some initial pressure disturbance $f(x)$ supported in $D$.  Data
is measured on the boundary of the body
\begin{equation}
g(x,t) = u(x,t)|_{\partial D \times \Rplus}.
\end{equation}
The inverse TAT problem is to reconstruct $f(x)$ from data $g(x,t).$
To ensure that $f$ might be determined by $g$ 
it  is essential that the solution $u(x,t)$ extend to a solution on 
all of $\R^d$.

It has become apparent that, to at least some extent, the measurements of an
ultrasound field acquired in thermoacoustic tomography determine the acoustic
properties of the body being imaged \cite{Wang:2006,ZA:2006,HKN:2008}. In
\cite{HKN:2008} it was proved for example that the thermoacoustic data of a
non-trivial source determines the speed from among its constant multiples. Numerical
work trying to recover both $f$ and $c$ can be found in
\cite{JYG:2006,Wang:2006,YaoJiang:2009,YuanJiang:2006,YuanJiang:2009,YZJ:2006,ZA:2006}.
A more substantial theoretical result is due to Stefanov and Uhlmann
\cite{SU:12a}, who proved that if the TAT data for two waves speeds and a very
special common initial value agree, and if the speeds and domain satisfy some
additional geometric hypotheses, then they must be equal.  In this paper we are
motivated to study a generalized interior transmission eigenvalue problem (ITP)
by looking for conditions on two speeds that ensure the TAT measurements they
generate are distinct.

To make the connection between the wave equation and the interior
transmission problem, we take the temporal Fourier transform of the
solution of the wave equation to get a family of Helmholtz
equations. Since we need to appeal to analyticity of this Fourier
transform, we must assume rapid decay in time of solutions of the wave
equation. This will restrict the classes of sound speed we can treat,
and will confine the discussion to odd dimensions.

Our uniqueness result asserts that, under suitable hypotheses on sound
speeds, the ranges of two thermoacoustic maps have trivial
intersection provided that the associated transmission eigenvalue
problem has not too many solutions. In particular, discrete spectrum
is sufficient.  However, to date we know of no results in the literature
which imply discrete transmission eigenvalue spectrum for smooth index
of refraction without assuming constant sign for the contrast. For the
special case of non-trapping radial sound speeds, we prove that in odd
dimension at least three, the transmission eigenvalue spectrum for a
distinct pair of sound speeds is discrete.

\section{Background and Notation}\label{sect:range}

Let $u$ be the solution of (\ref{forward_TAT_problem}) above. 
\begin{definition}\label{TAT_map}
The \emph{thermoacoustic map on $D$ with sound speed  $c(x)$} is defined by
\begin{equation}
\mathcal{L}_{c(x)} f(x) = u(x,t)|_{\partial D \times \Rplus}
\end{equation} 
for all initial pressure disturbances $f(x) \in C^{\infty}_0(D)$ and $u(x,t)$ a solution of
(\ref{forward_TAT_problem}). 
\end{definition}

For a given domain $D$ each sound speed defines a new thermoacoustic
map. In this paper we only consider a special class of sound speeds.

\begin{definition}\label{acoustic_profile}
  An \emph{acoustic profile on a domain $D$} is a smooth function
  $c(x) \in C^{\infty}(\mathbb{R}^d)$ such that
\begin{itemize}
\item[\emph{i})] $0 < \sigma < c(x) < \infty$ for $x \in D$  for some
  $\sigma > 0$ 
\item[\emph{ii})] $\textrm{supp}(1-c(x)) \subset D$.
\end{itemize}
\end{definition}

Properties of the range of $\mathcal{L}_{c(x)}$ are related to the
unique determination of the acoustic speed $c(x)$ by which we mean the
following. Given a non-zero function on $\partial D$ which lies in the
range of some thermoacoustic map, is there only one acoustic profile
$c$ for which it belongs to the range of $\mathcal{L}_c$?  We do not
address the question of recovering $c$ in the case when it is unique,
nor whether there is any stability estimate. In an analogous question
for the linearized forward map, Stefanov and Uhlmann \cite{SU:12b} have
recently proved an instability result. We refine the question of
unique determination by restricting the class of acoustic profiles. 

\begin{definition}\label{range_to_uniqueness}
  Let $\mathscr{D}$ be a set of acoustic profiles on some
  domain $D$. We say that a profile $c(x) \in \mathscr{D}$ is 
  uniquely determined in $\mathscr{D}$ by thermoacoustic data if and only if
  \begin{equation*}
    {\mathcal Rg}(\mathcal{L}_{c(x)}) \cap {\mathcal
      Rg}(\mathcal{L}_{b(x)}) = \{0\} 
  \end{equation*}
  for every $b(x) \in \mathscr{D}$ with $b(x) \ne c(x)$.
\end{definition}

The goal is to find conditions on $c(x)$ and $b(x)$ such that
\begin{equation}
  {\mathcal Rg}(\mathcal{L}_{c(x)}) \cap {\mathcal
    Rg}(\mathcal{L}_{b(x)}) = \{0\}. 
\end{equation}
To find these we prove results about the analyticity of the temporal
Fourier transform over the range of 
$\mathcal{L}_{c(x)}$. 

\begin{definition}
The \emph{temporal Fourier transform} of a function $u(x,t)$, with support in
$\mathbb{R}^d \times \Rplus$, is defined by 
\begin{equation*}
  \hat{u}(x, k) = \frac{1}{2 \pi} \int_0^{\infty} u(x,t) e^{i k t} \,dt. 
\end{equation*}
\end{definition}

The  next proposition follows directly from analyticity in a strip of the
Fourier transform of an exponentially decaying function.  

\begin{proposition}\label{Ranalytic}
  If the solution $u(x,t)$ of the forward thermoacoustic problem
  in domain $D \subset \mathbb{R}^d$ has  exponential
  decay in time uniformly over the closure of $D$, then there is an
  open strip in $\mathbb{C}$ containing
  $\mathbb{R}_+$ such that 
  for each fixed $x \in D$ the
  real part of the temporal Fourier transform, ${\mathcal
    Re}(\hat{u})(x, k)$, is the restriction to $\Rplus$ of a
  function analytic in the strip. 
\end{proposition}

Some sufficient conditions for exponential decay are known.

\begin{definition}
  An acoustic profile $c(x)$ is said to be \emph{non-trapping} if solutions, the
  \emph{bicharacteristics}, to
\begin{equation}
\left\{ \begin{array}{c}
\dot{x} = c^2 (x) \xi \,\,\,\,\,\,\,\,\,\,\,\,\,\,\,\,\,\,\,\,\\
\dot{\xi} =-\frac{1}{2} \nabla(c^2(x)) |\xi|^2 \\
x(0) = x_0,\, \xi(0) = \xi_0,
\end{array} \right.
\end{equation}
in $\mathbb{R}_{x,\xi}^{2n}$ have projections, \emph{rays}, in
$\mathbb{R}^d_x$ tending to infinity as $t \rightarrow 
\infty$ as long as $\xi_0 \ne 0.$
\end{definition}
If this holds, one has the following theorem of Vainberg.
\cite{Vainberg:1975,Vainberg:1989}. 
\begin{theorem}\label{vainberg_decay}
  If the non-trapping condition is satisfied, then for any multi-index
  $\alpha = (\alpha_0, \alpha_1, \dots, \alpha_d),$ the following
  estimate holds for solutions of the thermoacoustic forward problem:
  \begin{equation*}
    \left| \partial^{\alpha}_{(t,x)} u(x,t) \right| \le C \eta(t) \| f
    \|_{L^2}, \,\, x \in D.
  \end{equation*}
  Here the function $\eta(t)$ that characterizes the decay is
  $t^{-d-\alpha_0+1}$ when the dimension $d$ is even and $e^{-\epsilon
    t}$ when $d \geq 3$ is odd.
\end{theorem}

By the theorem of Vainberg, analyticity of the temporal Fourier
transform in a neighborhood of the
positive real axis  will hold for non-trapping speeds in odd
dimensions greater than one. In the remainder of this paper, we shall
implicitly assume all acoustic profiles are non-trapping. Non-trapping
also ensures that for a single sound speed $c$, the thermoacoustic map
$\mathcal{L}_c$ is injective, \cite{FR:2009, SU:2009}.

\section{Relation of TAT to the Interior Transmission
  Problem}\label{sect:ITP2uniq} 

The following relation of the wave equation to the Helmholtz equation
is standard.
\begin{proposition}\label{inhomogen_Helm}
Let $u(x,t)$ satisfy
\begin{align}\label{fwd2}
\partial^2_t u(x,t) - c^2(x) \Delta u(x,t) &= 0 \textrm{ in }
\mathbb{R}^d \times \Rplus \\ \nonumber 
u(x,0) = f(x),\,\, \partial_t u(x,t) &= 0 \textrm{ on } \mathbb{R}^d \nonumber
\end{align}
and set $n(x) = \frac{1}{c^2(x)}$. For $k \in \mathbb{R}_+$ the
temporal Fourier transform $\hat{u}(x,k)$ satisfies 
\begin{equation}\label{inhom_Helm_eqtn}
\Delta \hat{u}(x,k) + k^2 n(x) \hat{u}(x,k) = \frac{i k}{2 \pi} n(x)
f(x) \textrm{ for } x \in \mathbb{R}^d
\end{equation}
and $U(x,k) = {\mathcal Re}(\hat{u})(x,k)$ satisfies
\begin{equation}\label{Helm_eqtn}
\Delta U(x,k) + k^2 n(x) U(x,k) = 0 \textrm{ for } x \in \mathbb{R}^d.
\end{equation}
\end{proposition}

We now suppose that $c$ and $b$ are two acoustic profiles, and
$f_1,f_2$ two initial conditions supported in $D$ such that
$\mathcal{L}_{c(x)}(f_1)=\mathcal{L}_{b(x)}(f_2).$ Let $u,v$
denote the respective solutions of (\ref{forward_TAT_problem}). Then
as $c=b=1$ and $f_1=f_2=0$ in $\R^d \setminus D$, $u,v$ are solutions
of the (same) wave equation in the exterior domain, with the same
boundary value on $\partial D \times \R_+$ and with the same (zero) initial
conditions. Since the exterior initial boundary-value problem is
well-posed, the solutions are equal in the exterior domain. Let
$\hat{u}$ and $\hat{v}$ be the temporal Fourier transforms. This
implies that 
$\hat{u}(x,k) =\hat{v}(x,k)$ in $\R^d \setminus D$ for all $k$ so their
normal derivatives on $\partial D$ are equal. Since by hypothesis,
$u=v$ on $\partial D \times \Rplus,$ then 
$\hat{u}(x,k) = \hat{v}(x,k)$ for $x\in \partial D$.
Then $U = \Re(\hat{u})$ and $V = \Re(\hat{v})$
are solutions of the following problem.  

\begin{definition}
  A wavenumber $k\geq 0$ is called a \emph{transmission
    eigenvalue} if there 
  exists an non-trivial pair $(u,v) \in H^2(D) \times H^2(D)$ solving 
  the \emph{interior transmission problem} (ITP) relative
  to the acoustic profiles $c(x)$ and $b(x)$ in $D$ if
\begin{align}\label{interior_transmission}
\Delta u + k^2 n_c(x) u = 0 &\textrm{ in } D \\ \nonumber
\Delta v + k^2 n_b(x) v = 0 &\textrm{ in } D \\ \nonumber
u = v, \, \partial_{\nu} u = \partial_{\nu} v &\textrm{ on } \partial
D. \\ \nonumber 
\end{align}
Here, $\partial_{\nu}$ represents the outward normal derivative on
$\partial D$, $n_c(x)=c^{-2}(x)$ and $n_b(x)$ is defined similarly. 
\end{definition}

\begin{definition}
The \emph{real transmission eigenvalue spectrum} is the set of
non-negative real transmission eigenvalues.
\end{definition}

\begin{remark}
The interior transmission problem arose in scattering theory. In that
setting, one of the sound speeds is usually taken to be constant, but
here it is natural to assume that both are variable. Researchers in
scattering theory have profitably considered complex transmission
eigenvalues, but they play no role here. 
\end{remark}

The next theorem shows range separation conditions may be derived from
sparseness of the interior transmission spectrum. 

\begin{theorem}\label{clust_pt_spect_implies_uniq}
  Let $c(x)$ and $b(x)$ be non-trapping acoustic profiles in a domain $D$. If
  the complement of the interior transmission spectrum has a finite cluster
  point in $\Rplus$ then the intersection of the range of the thermoacoustic
  operators, $\mathcal{L}_{c(x)}$ and $\mathcal{L}_{b(x)}$ reduces to zero.
\end{theorem}

\begin{proof} Let $u,v$ be the solutions corresponding to $f_1, f_2$
  with $\mathcal{L}_{c} f_1 = \mathcal{L}_{b} f_2.$  We have
  already observed that $U(x,k) = \Re(\hat{u})$ and 
 $V(x,k) =\Re(\hat{v})$ satisfy (\ref{interior_transmission}) for every $k\in
  \Rplus.$ If $k$ is not a transmission eigenvalue, then the only
  solution to (\ref{interior_transmission}) is the zero function,
  hence for such $k$, $U(x,k)=V(x,k)=0$ in $D$. Since by lemma
  \ref{Ranalytic} both $U$ and $V$ are real 
  analytic on the positive real axis,  if they are zero on set with a finite
  accumulation point, they are identically zero. Then $u,v$ must be zero.
\end{proof}

Therefore, any result which implies that the real transmission spectrum
associated to the pair $n_c,n_b$ is discrete, or has a positive lower
bound will imply that the ranges of the corresponding thermoacoustic
operators have only trivial intersection. If this holds for every pair
in a class $\mathcal{D}$, then acoustic profiles are uniquely
determined within the class.

On a domain $D$, let $n_c$ and $n_b$ be associated to acoustic
profiles $c$ and $b$. Define the \emph{contrast} to be
the difference:
\begin{equation*}
m_{cb}(x) = n_c(x) - n_b(x)
\end{equation*}
and note $\emph{supp}(m_{cb}) \subset \bar{D}.$ There are many results
in the literature giving conditions on $m_{cb}$ which are sufficient
to guarantee discrete spectrum, or a spectral gap. These are usually
insufficient for our purposes however, since in most cases there is
an assumption that $m_{cb}$ is discontinuous at some interface, while
we require that $m_{cb}$ be smooth. 

One easy result is the following theorem, which is similar to the lower
bound in \cite{CPS:2008}, for the case when $m_{cb}$ is either
non-negative or non-positive. Let $\lambda_0$ be the first eigenvalue
of the Dirichlet Laplacian in $D$ and set $n_i^* = \sup_{D} n_i(x)$
for $i = c,\,b.$ 

\begin{theorem}\label{ITspect_low_bnd}
If $k \in \Rplus$ is a transmission eigenvalue and $m(x) = n_c(x) -
n_b(x) \ge 0$ then  
\begin{equation*}
k \ge \sqrt{\frac{\lambda_0}{n_c^*}}.
\end{equation*}
If $m(x) = n_c(x) - n_b(x) \le 0$ then 
\begin{equation*}
k \ge \sqrt{\frac{\lambda_0}{n_b^*}}.
\end{equation*}
\end{theorem}
We omit the proof.

\begin{theorem}\label{partialuniq}
  Consider two profiles, $c(x)$ and $b(x)$, relative to a domain $D$. If 
  \begin{equation*}
    c(x)-b(x) \ge 0 \textrm{ or } c(x)-b(x) \le 0
  \end{equation*}
  in $D$ then the thermoacoustic data generated by the domain $D$ from
  the acoustics $c(x)$ cannot be generated by the 
  domain $D$ with the profile $b(x)$. That is, the intersection of the
  ranges of the operators $\mathcal{L}_{c(x)}$ and 
  $\mathcal{L}_{b(x)}$ is zero for any two acoustic profiles whose
  difference does not change signs in $D$. 
\end{theorem}

\begin{proof}
  The proof follows by noticing that $c(x) - b(x)$ not changing signs
  in $D$ is equivalent to $m(x)$ not changing signs in 
  $D$. Then one appeals to theorem \ref{ITspect_low_bnd} and theorem
  \ref{clust_pt_spect_implies_uniq}. 
\end{proof}

\begin{corollary}
  Suppose thermoacoustic data $h(x,t)$ on $\partial D \times \Rplus$
  is generated by an acoustic profile in some set 
  $\mathscr{D}$. Assume also that for every pair $c(x),\, b(x) \in \mathscr{D}$
  \begin{equation*}
    c(x) - b(x) \ge 0 \textrm{ or } c(x) - b(x) \le 0
  \end{equation*}
  on $D$. Then the acoustic profile generating data $h(x,t)$ is
  determined uniquely in $\mathscr{D}$. 
\end{corollary}

\section{ Radially Symmetric ITP}

In this section we restrict to  the class $\mathcal{D}_S$ of radially
symmetric acoustic profiles in a ball. We prove sparsity of the transmission
spectrum for a distinct pair from $\mathcal{D}_S.$  

We assume $d \geq 3$ is odd, and assume that $D$ is the unit ball
$B_1$. Let $S_1$ be the unit sphere.  The first step in the study is
to reduce the problem to an ordinary differential equation. Following
\cite{CM:1988, MPS:1994, MP:1994, CK:2013} the Helmholtz equation with
radial sound speed can be separated in spherical coordinates as a sum
of products of radial functions with spherical harmonics, where the
radial components satisfy an ordinary differential equation depending
on the degree of the spherical harmonic.  Using the independence of
spherical harmonics, a pair of solutions corresponding to two radial
acoustic profiles satisfies (\ref{interior_transmission}) for some $k^2$
if and only if the radial components corresponding to some spherical
harmonic have the same Cauchy data at $r=1.$ It is convenient to make
a Liouville transformation on the resulting radial differential
equations.  We summarize with the following lemmas.

\begin{lemma}\label{COV_spherical_harmonic_expansion}
  For a radially symmetric, smooth, refractive index $n(r)$ a solution $w(r,\theta)$ of 
  \begin{equation}\label{var_coef_helmholtz}
    \Delta w + k^2 n(r) w = 0
  \end{equation}
  is given by 
  \begin{equation}\label{spherical_expansion}
    w(r,\theta) = \sum_{j=0}^{\infty} \sum_{l = 0}^M f_{jl}(r) r^j Y_{jl}(\theta).
  \end{equation}
  The coefficient functions, $f_{jl}(r)$, satisfy the Sturm-Liouville equation
  \begin{equation}\label{STL1}
    \partial_r(r^{\gamma} \partial_r f_{jl}) + k^2 n(r) r^{\gamma}
    f_{jl} = 0 \textrm{ on } [0,1], 
  \end{equation} 
  where $\gamma = d+2j-1.$ The $f_{jl}(r)$  are bounded at $r=0$ and
  so are constant multiples of single such $f_j.$ Set $m =
  \frac{\gamma}{2}-1 = j 
  +\frac{d-3}{2}.$ Defining $X_{m}$ by the Liouville transformation,  
  \begin{equation}\label{Ltrans1}
    \eta = \int_0^r \sqrt{n(\sigma)} \,d\sigma,\, X_{m}(\eta) =
         [r^{2\gamma}n(r)]^{1/4}f_{j}(r) 
  \end{equation}
  then $X_{m}(\eta)$ satisfies
  \begin{equation}\label{nonrad_SL_eigenvalue_problem}
    -\ddot{X}_{m}(\eta) + \left(\frac{m(m+1)}{\eta^2} +
    p_m(\eta)\right) X_{m}(\eta) = k^2 X_{m}(\eta)  
    \textrm{ for } \eta \in [0,C],
  \end{equation}
  where
  \begin{equation}\label{Ltrans2}
    p_m(\eta) = \frac{1}{4}\frac{\ddot{n}(r)}{(n(r))^2} -
    \frac{5}{16}\frac{(\dot{n}(r))^2}{(n(r))^3}  
    + m(m+1)\left( \frac{1}{r^2 n(r)} - \frac{1}{\eta^2} \right)
  \end{equation}
  and $C = \int_0^1 \sqrt{n(s)} \,ds.$
\end{lemma}

Note that, in the summation (\ref{spherical_expansion}), the index $M$
depends on the order $j$ and dimension $d$. The specific value of $M$ is not
important to our results, it is the dimension of the space of spherical
harmonics of order $j$ in dimension $d$.

\begin{proof} 
Everything is standard, except perhaps the form $f_{jl}(r)r^j$ which
results from the smoothness of $w$ and the orthogonality of the
spherical harmonics to lower degree terms in the Taylor expansion of
$w.$ The independence (up to constant) of $f_{jl}$ on $l$ results from the
singular S-L equation 
(\ref{STL1}) having a unique (normalized) bounded solution, see
\cite{GR:1988,Car:1993,Car:1997}. 
\end{proof}

The proof of the following lemma is included in the appendix, section
\ref{p_m_bounded_pf}. 
\begin{lemma}\label{p_m_bounded}
  For a radial acoustic profile $c(r)$ the coefficient function (\ref{Ltrans2})
  with $n(r) = c^{-2}(r)$ is bounded on the interval $[0,C],$ $C = \eta(1).$ 
\end{lemma}

We will use $X_{m}(\eta)$ to denote a solution of equation
(\ref{nonrad_SL_eigenvalue_problem}). The boundedness of $f_j$ at $r=0$
imposes a boundary condition on the $X_{m}(\eta)$ at $\eta = 0,$ namely
\begin{equation}\label{init_cond}
\lim_{\eta \rightarrow 0} \eta^{-(m+1)} X_{m}(\eta) < \infty,
\end{equation} 
since $\frac{\eta}{r}$ has a finite positive limit at $r=0.$

Following \cite{GR:1988,Car:1993,Car:1997}, 
(\ref{nonrad_SL_eigenvalue_problem}) has a fundamental set of
solutions $X_{m1}, X_{m2}$ satisfying 
\begin{equation}
\lim_{\eta \rightarrow 0} \eta^{-(m+1)} X_{m1}(\eta,k) = 1,
\,\lim_{\eta \rightarrow 0} \eta^{m} X_{m2}(\eta,k) = 1. 
\end{equation}
From (\ref{init_cond}) we see that $X_{m}(\eta,k)$ must be a constant
  multiple of $X_{m1}.$ 

We now consider two radially symmetric acoustic profiles.  If $k$ is a
transmission eigenvalue for the pair on $B_1$, then equality of the
Cauchy data of the corresponding Helmholtz equations implies equality
of each spherical harmonic component of the Cauchy data, and
non-triviality of the solution implies that some term in the spherical
harmonic expansion is non-trivial. Combined with the respective
Liouville transformations, this proves most of the following result.
\begin{lemma}\label{nonradial_ITP_reformulation}
  For two radially symmetric acoustic profiles $c(r)$ and $b(r)$ on
  $B_1(0) \subset \mathbb{R}^d$ the transmission 
  spectrum is equal to the set of all $k \in \Rplus$ such that, for
  some $m = j + \frac{1}{2}(d-3),$ there exists 
  non-trivial solutions $(X_m(\eta,k), Z_m(\xi,k)) \in C^2((0,C])
    \times C^2((0,B])$ satisfying 
  \begin{align}\label{nonradial_ITP_eqtn}
    \ddot{X}_m(\eta, k) + \left(k^2 - \frac{m(m+1)}{\eta^2} \right)
    X_m(\eta, k) &= p_{1m}(\eta) X_m(\eta, k),\, \eta \in [0,C]
    \\ \nonumber 
    \ddot{Z}_m(\xi, k) + \left(k^2 - \frac{m(m+1)}{\xi^2} \right)
    Z_m(\xi, k) &= p_{2m}(\xi) Z_m(\xi, k),\, \xi \in [0,B]
    \\ \nonumber 
    \lim_{\eta \rightarrow 0} \eta^{-(m+1)} X_m(\eta, k) < \infty, &\,
    \lim_{\xi \rightarrow 0} \xi^{-(m+1)} Z_m(\xi, k) < \infty
    \\ \nonumber 
    X_m(C, k) = Z_m(B, k), \, \dot{X}_m(C, k) &= \dot{Z}_m(B, k).
  \end{align} 
\end{lemma}
In the statement of the lemma, recall that $n_c(r) = c^{-2}(r)$, $n_b(r) =
b^{-2}(r)$. The Liouville transform (\ref{Ltrans1}) then yields the coefficients
$p_{1m}(\eta)$ and $p_{2m}(\xi)$ defined in (\ref{Ltrans2}). Here $\xi = \xi(r)$
is the new independent variable introduced by (\ref{Ltrans1}) using
$n_b(r)$. The right endpoints are defined by $C = \int_0^1 \sqrt{n_c(s)} \,ds$
and $B = \int_0^1 \sqrt{n_b(s)} \,ds$.

\begin{proof}
  The only matter left to check is the equality of $X_m, Z_m$ and
  their derivative at the endpoints of their respective
  intervals. This just requires tracing through their definition by
  the Liouville transform. 
\end{proof}

\begin{remark}
  The condition above depends on $m= j + \frac{1}{2}(d-3)$. Thus the
  transmission spectrum in odd dimensions greater than three is a subset of the
  transmission spectrum in dimension three.  (However, the multiplicity of each
  transmission eigenvalue is greater, since the dimension of the space of
  spherical harmonics grows with dimension.) Any result which implies sparseness
  in dimension three implies sparseness in higher odd dimensions.
\end{remark}

Let $X_{m1}(\eta,k)$ and $Z_{m1}(\xi,k)$ be the solutions of
(\ref{nonradial_ITP_eqtn}) corresponding to $c(r)$ and $b(r),$
respectively, satisfying
\begin{equation*}
\lim_{\eta \rightarrow 0} \eta^{-(m+1)} X_{m1}(\eta, k) = 1 \textrm{ and } 
\lim_{\xi \rightarrow 0} \xi^{-(m+1)} Z_{m1}(\xi, k) = 1.
\end{equation*}
We denote the Wronskian of two fundamental solutions evaluated at
different endpoints by  
\begin{equation*}
W(Z_{m1}(B,k), X_{m1}(C,k)) = Z_{m1}(B,k) \dot{X}_{m1}(C,k) -
\dot{Z}_{m1}(B,k) X_{m1}(C,k). 
\end{equation*}

\begin{corollary}\label{nonradial_det_cond_thm}
  Let $d \geq 3$ be odd. For two radially symmetric acoustic speeds
  $c(r)$ and $b(r)$ on $B_1(0) \subset \mathbb{R}^d$, $k \in \Rplus$
  is a transmission eigenvalue if and only if
  \begin{equation}
    W(Z_{m1}(B,k), X_{m1}(C,k)) = 0 
  \end{equation}
  for some integer $m \ge \frac{1}{2}(d-3)$.
\end{corollary}

\begin{proof}  
  By Theorem \ref{nonradial_ITP_reformulation}, $k$ is transmission
  eigenvalue if and only if there exists a non-trivial pair $X_m, Z_m$ as in the
  theorem. These must be multiples of $X_{m1}, Z_{m1}$ by the
  condition at $0.$ The equalities at the right endpoints imply
  $\alpha X_{m1}(C)= \beta Z_{m1}(B)$ and similarly for the
  derivatives, which is a linear system for $\alpha, \beta$  which has
  a non-trivial 
  solution if and only if the 
  Wronskian condition holds. 
\end{proof}

We will often use the shorter notation
\begin{equation*}
d_m(k) = Z_{m1}(B) \dot{X}_{m1}(C) - \dot{Z}_{m1}(B) X_{m1}(C)
\end{equation*}
As in the case of the usual transmission eigenvalue
problem, the determinant is the restriction to the positive real axis
of an entire function of exponential type. We note the following
\begin{proposition}\label{uncount}
  The transmission spectrum associated to radial acoustic profiles $c(r),b(r)$
  is uncountable if and only if for some $m$, $W(Z_{m}(B),X_{m}(C))$ is
  identically zero as a function of $k.$
\end{proposition}
\begin{proof} If $d_m(k)$ is identically zero, then every $k$
  is a transmission eigenvalue, since $X_m$ and $Z_m$ are non-trivial.
  Conversely, if the transmission spectrum is uncountable, then the
  transmission spectrum is uncountable for some $m$, and so $d_m(k)$
  determinant vanishes on an uncountable set. Since $d_m$ is analytic, it
  vanishes identically.
\end{proof}

Let us now assume that the transmission spectrum associated to radial sound
speeds $b(r),c(r)$ is uncountable. Our main result of this section is the
following theorem.
\begin{theorem}\label{main_rad} 
  Let $b(r),c(r)$ be radial sound speeds with $n_b$, $n_c$ the associated index
  of refraction. If the transmission spectrum for the pair $n_b,n_c$ is
  uncountable, then $n_c=n_b.$
\end{theorem}

\begin{proof}
  The proof has two steps.  The first is to establish that $B=C;$ that is, that
  the interval $[0,1]$ has the same length in the respective slowness
  metrics. In the case of the standard transmission eigenvalue problem when one
  of the speeds is unity, there is an asymptotic expression for the determinant
  (\cite{CK:2013}, eq. (8.38)) which implies discreteness of the zeros when
  $C\neq 1$ (taking $b=1$). The proof in the general situation is slightly more
  complicated, but follows the same line making use of the asymptotics of the
  solutions resulting from their expression as analytic functions of the
  potentials.  We omit the details here: they can be found in \cite{KSH:10}.

By proposition \ref{uncount}, we may now assume that for some fixed $m$, the
determinant $d_m$ is identically zero. We will select some particular
values of $k$. Let $k_1^2$ be a Dirichlet eigenvalue for
(\ref{nonrad_SL_eigenvalue_problem}) for $p = p_{1m}$ on $[0,C].$
Since the space 
of solutions satisfying the boundary condition at $\eta =0$ is one 
dimensional, $X_m$ must be a Dirichlet eigenfunction. Since $X_m$ has
the same Cauchy data at $\eta =C$ as $Z_m$, then $Z_m$ must  also be
a Dirichlet eigenfunction and so $k_1^2$ is a Dirichlet eigenvalue for 
(\ref{nonrad_SL_eigenvalue_problem}) for $q = p_{2m}$ on $[0,C].$
Reversing the roles shows that potentials $p$ and $q$ have the same
Dirichlet spectrum. A similar argument shows that if $k^2$ is chosen
from the Neumann spectrum of $p$ ($X_m'(C)=0$ with the boundedness
condition at $\eta =0$) then $k^2$ must also belong to the Neumann
spectrum for $q$, and conversely. However, it is known
(\cite{Car:1997}, Theorem 1.3)
that equality of spectra for the Bessel type operator for two
independent boundary conditions implies equality of the
potentials. Thus $p_{1m} = p_{2m}$ on $[0,C].$ The proof is finished
by applying the following proposition.
\end{proof}

\begin{proposition} Let $p_{1m}, p_{2m}$ be the coefficients defined through
  (\ref{Ltrans1})-(\ref{Ltrans2}) on $[0,C]$. Also, assume that $[0,1]$ has the
  same length with respect to the two metrics defined by $n_c$ and $n_b$, that
  is assume $B = C$. If $p_{1m}=p_{2m}$ then $n_c = n_b.$ 
\end{proposition}
\begin{proof}
  The proof is a minor modification of the proof of Theorem 4.3 in
  \cite{Som:95}, which treated the case $m=1.$ Since \cite{Som:95} is
  not easily accessible (some of its results were summarized in
  \cite{MSS:1997}), we present it here. Denote by $r_1(\eta)$ and
  $r_2(\xi)$ the inverses of $\eta(r)$ and $\xi(r)$ respectively, and
  replace $n_c,n_b$ by $n_1,n_2$ respectively. The equality of
  $p_{1m}$ and $p_{2m}$ is to be interpreted as equality when the
  first is evaluated at $r=r_1(\eta)$ and the second at
  $r=r_2(\eta)$. Here both $r_1$ and $r_2$ may be evaluated at $\eta$
  since $B = C$ implies $r_1$ and $r_2$ are defined over the same
  domain. Now, for $i=1,2$ define $u_i(\eta)$ by
  \[ u_i(\eta) = \frac{n_i^{1/4}(r_i(\eta))}{r_i(\eta)^m}.\] 
  Then a calculation shows that $u_i$ satisfies
\begin{align*}
\ddot{u_i}(\eta) - \left(p_{im}(\eta) + \frac{m(m+1)}{\eta^2}\right) u_i &=
0, \quad 0 < \eta \leq C,\\
u_i(C) = C^{-m}& \qquad 
\dot{u_i}(C) = -m C^{-(m+1)},
\end{align*}
where we have used that $n_i(1) =1$ and $n_i'(1)=0.$ Applying the
uniqueness theorem for ordinary differential equations, it holds that
$u_1(\eta) = u_2(\eta)$ on $(0,C].$ Taking reciprocals and squaring
  gives that $(r_1^{2m+1} -r_2^{2m+1})'=0$ using that $r_i'
  =\frac{1}{\sqrt{n_i}}.$ Since $r_1(C)=1=r_2(C)$, the difference is
    zero on the interval. Taking $(2m+1)^{th}$  roots gives $r_1 = r_2$
    and differentiation finishes the proof. 
\end{proof}

\begin{remark}
Theorem 2.1 of  \cite{CMS:2010} has a similar statement to
Theorem \ref{main_rad}, but an examination of the proof shows that the
hypothesis is really that $d_m$ is identically zero for all $m$.
\end{remark}

The following theorem and corollary summarizes the implications of the results
of this section for the TAT problem.

\begin{theorem}\label{TATresult_summary1}
  Let $c(r)$ and $b(r)$ be non-trapping, radially symmetric, acoustic profiles
  in the unit ball, $B_1 \subset \mathbb{R}^d$ with $d \ge 3$ odd. If $c(r) \ne
  b(r)$ then the intersection of the range of the thermoacoustic
  operators, $\mathcal{L}_{c(r)}$ and $\mathcal{L}_{b(r)}$ reduces to zero. 
\end{theorem}

\begin{corollary}\label{TATresult_summary2}
  Suppose $d \ge 3$ is odd and that thermoacoustic data $h(x,t)$ on $\partial
  B_1 \times \Rplus$ is generated by a radially symmetric, non-trapping acoustic
  profile. Then the acoustic profile generating data $h(x,t)$ is uniquely
  determined among the set of radially symmetric, non-trapping acoustic
  profiles.
\end{corollary}

\section{Conclusion}

This work details a relation between the unique determination of the
acoustic profile of a body from thermoacoustic data and properties of
the spectrum of the interior transmission eigenvalue problem. The
difference of two acoustic profiles gives a constrast which does not
realistically satisfy the constant sign or coercivity hypotheses of
previous work on the transmission eigenvalue problem. Radial
(non-trapping) sound speeds give the simplest examples of such
profiles, and we have succeeded to show that the associated spectrum
of pairs within this class is discrete. It would be of great interest to
develop a general method to analyze the transmission spectrum for two
acoustic profiles without positivity assumptions or radial symmetry
assumptions.

Our analysis does not put any \emph{a
  priori} conditions on the allowable initial impulses $f(x)$ besides
$\textrm{supp}(f) \subset D.$ Since we study the uniqueness question
in terms of separation of ranges of operators $\mathcal{L}_{c(x)}$,
restriction of the domain might lead to new uniqueness results. In
particular, focussed initial impulses might offer the possibility of
both uniqueness and inversion.
\section{Appendix}

\subsection{Proof of boundedness of $p_m(\eta)$}\label{p_m_bounded_pf}

Here we prove lemma \ref{p_m_bounded}. First, notice that since $c(r)$
is an acoustic profile it must be smooth, bounded, even in $r$, and
bounded away from zero. Thus, $n(r)$ is smooth, bounded, even in $r$,
and bounded away from zero. Moreover all derivatives of $c(r)$ and
$n(r)$ must also be bounded. This implies that the first two terms in
$p_m(\eta)$ are bounded on $[0,C]$. It remains to show that the term
\begin{equation*}
  \frac{1}{r^2 n(r)} - \frac{1}{\eta^2(r)}
\end{equation*}
is bounded on $[0,C],$ for which it suffices to prove that it has a
finite limit at $0.$
Writing the difference
\begin{equation*}
\frac{1}{r^2 n(r)}- \frac{1}{\eta^2(r)}= \frac{1}{n(r)} \left(1+
\sqrt{n}\frac{r}{\eta}\right) \left( \frac{\eta
  -\sqrt{n}r}{r^2\eta}\right).
\end{equation*}
The second term on the right has limit equal to 2, while integration
by parts in the definition of $\eta$ gives that $\eta = \sqrt{n}r
-\int_0^r s \phi'(s) \, ds,$ where $\phi= \sqrt{n} = \frac{1}{c}.$
This simplifies the numerator of the last term on the right so that an
application of l'Hospital's rule conveniently shows that its limit is
$-2\frac{\phi''(0)}{\eta'(0)}.$ 

\section{Acknowledgements}
Some of the results in this paper first appeared in the PhD thesis of
K. Hickmann at Oregon State University \cite{KSH:10}. David Finch
gratefully acknowledges support from NSF grant DMS 100914.

\bibliographystyle{amsplain}
\bibliography{trans_tat}

\end{document}